\documentclass[letterpaper,12pt,peerreviewca,draftcls]{IEEEtran}
\usepackage{csm16}
\usepackage[margin=1in]{geometry}
\usepackage[nolists,nomarkers,tablesfirst]{endfloat} 
\usepackage{amsmath,amssymb}
\usepackage{enumerate}
\usepackage{bm}
\newtheorem{theorem}{Theorem}   
\newtheorem{lemma}{Lemma}

\newtheorem{proposition}{Proposition}
\newtheorem{corollary}{Corollary}
\newtheorem{example}{Example}
\newtheorem{remark}{Remark}

\newtheorem{definition}{Definition}
\newtheorem{assumption}{Assumption}

\newcommand{\bmv}{\bm{v}}

\DeclareMathOperator{\In}{In}

\DeclareMathOperator{\inter}{int}

\newcommand{\None}{
\left[\begin{array}{c|c}
I & \begin{array}{c}
X_+\\Y_- 
\end{array}
\\\hline
0 & \begin{array}{c}
-X_-\\-U_- 
\end{array}
\end{array}\right]
\!
\bbm
\Phi_{11} & \Phi_{12}\\
\Phi_{12}^\top & \Phi_{22}
\ebm\!
\left[\begin{array}{c|c}
I & \begin{array}{c}
X_+\\Y_- 
\end{array}
\\\hline
0 & \begin{array}{c}
-X_-\\-U_- 
\end{array}
\end{array}\right]^\top\!\!
}

\newcommand{\systwo}{
\bbm
I\\\hline\\[-3mm]
\begin{matrix}
A^\top & C^\top\!\\
B^\top & D^\top\!
\end{matrix}
\ebm
}

\newcommand{\sysone}{
\bbm
I\\\hline\\[-3mm]
\begin{matrix}
A & B\\
C & D
\end{matrix}
\ebm
}

\usepackage{url}
\usepackage{graphicx,xcolor}
\usepackage{verbatim}
\makeatletter
\newcommand{\verbatimfont}[1]{\def\verbatim@font{#1}}%
\makeatother
\verbatimfont{\ttfamily\small}

\newcommand{\bi}{\begin{itemize}}\newcommand{\ei}{\end{itemize}}
\newcommand{\be}{\begin{equation}}\newcommand{\ee}{\end{equation}}
\newcommand{\bee}{\begin{enumerate}}\newcommand{\eee}{\end{enumerate}}
\newcommand{\bea}{\begin{eqnarray}}\newcommand{\eea}{\end{eqnarray}}
\newcommand{\beas}{\begin{eqnarray*}}\newcommand{\eeas}{\end{eqnarray*}}
\newcommand{\bc}{\begin{center}}\newcommand{\ec}{\end{center}}
\usepackage[left,pagewise]{lineno} 
\usepackage[english]{babel} 
\usepackage{blindtext}

\title{Data-driven dissipativity analysis:\\
\Large application of the matrix S-lemma}
\author{Henk J. van Waarde, M. Kanat Camlibel, Paolo Rapisarda, and Harry L. Trentelman\\
	POC: H. J. van Waarde (hv280@cam.ac.uk)\\ \today}

\newif\ifPDF \ifx\pdfoutput\undefined\PDFfalse \else\ifnum\pdfoutput > 0\PDFtrue \else\PDFfalse \fi \fi
\ifPDF 
\usepackage[pdftex, plainpages = false, colorlinks=true, linkcolor=black, citecolor = green!50!blue, urlcolor = blue, filecolor=black, pagebackref=false, hypertexnames=false,  pdfpagelabels ]{hyperref}
\fi

\usepackage{comment}

\DeclareMathOperator{\rank}{rank}

\let\leq\leqslant
\let\geq\geqslant

\newcommand{\calN}{\ensuremath{\mathcal{N}}}

\newcommand{\hatu}{\ensuremath{\hat{u}}}

\newcommand{\haty}{\ensuremath{\hat{y}}}

\newcommand{\hatF}{\ensuremath{\hat{F}}}
\newcommand{\hatG}{\ensuremath{\hat{G}}}
\newcommand{\hatH}{\ensuremath{\hat{H}}}

\newcommand{\bmat}{\begin{matrix}}
\newcommand{\emat}{\end{matrix}}
\newcommand{\bbm}{\begin{bmatrix}}
\newcommand{\ebm}{\end{bmatrix}}
\newcommand{\bbma}{\begin{bmatrix*}[r]}
\newcommand{\ebma}{\end{bmatrix*}}
\newcommand{\bpm}{\begin{pmatrix}}
\newcommand{\epm}{\end{pmatrix}}
\newcommand{\bvm}{\begin{vmatrix}}
\newcommand{\evm}{\end{vmatrix}}
\newcommand{\bse}{\begin{subequations}}
\newcommand{\ese}{\end{subequations}}
\newcommand{\beq}{\begin{equation}}
\newcommand{\eeq}{\end{equation}}
\newcommand{\ben}{\renewcommand{\labelenumi}{\arabic{enumi}.}
\renewcommand{\theenumi}{\arabic{enumi}}\begin{enumerate}}

\newcommand{\een}{\end{enumerate}}

\newcommand{\beni}{\renewcommand{\labelenumi}{\roman{enumi}.}
\renewcommand{\theenumi}{\roman{enumi}}\begin{enumerate}}

\newcommand{\eeni}{\end{enumerate}}

\newcommand{\bena}{\renewcommand{\labelenumi}{\alph{enumi}.}
\renewcommand{\theenumi}{\alph{enumi}}\begin{enumerate}}

\newcommand{\eena}{\end{enumerate}}

\newcommand{\bit}{\begin{itemize}}
\newcommand{\eit}{\end{itemize}}
\newcommand{\bthe}{\begin{theorem}}
\newcommand{\ethe}{\end{theorem}}
\newcommand{\blem}{\begin{lemma}}
\newcommand{\elem}{\end{lemma}}
\newcommand{\bprop}{\begin{proposition}}
\newcommand{\eprop}{\end{proposition}}
\newcommand{\bex}{\begin{example}}
\newcommand{\eex}{\end{example}}
\newcommand{\bas}{\begin{assumption}}
\newcommand{\eas}{\end{assumption}}
\newcommand{\bre}{\begin{remark}}
\newcommand{\ere}{\end{remark}}
\newcommand{\bcor}{\begin{corollary}}
\newcommand{\ecor}{\end{corollary}}
\newcommand{\bdfn}{\begin{definition}}
\newcommand{\edfn}{\end{definition}}
\newcommand{\bcon}{\begin{conjecture}}
\newcommand{\econ}{\end{conjecture}}

\newcommand{\inv}{\ensuremath{^{-1}}}
\newcommand{\pset}[1]{\ensuremath{\{#1\}}}

\newcommand{\zset}{\ensuremath{\pset{0}}}

\newcommand{\norm}[1]{\ensuremath{\| #1 \|}}

\newcommand{\R}{\ensuremath{\mathbb R}}
\newcommand{\C}{\ensuremath{\mathbb C}}

\newcommand{\qand}{\quad\text{ and }\quad}

\begin{document}
\maketitle
\CSMsetup
\nolinenumbers \modulolinenumbers[2] 

As is generally acknowledged, the work of Jan Willems on dissipativity has formed the foundation for large parts of systems and control theory as developed in the past fifty years. While working as an assistant professor in the Department of Electrical Engineering at MIT during the period between 1968 to 1973, he wrote the ground-breaking papers ‘‘Dissipative dynamical systems, General theory’’; and ‘‘Dissipative dynamical systems, Linear systems with quadratic supply rates’’ in Archive for Rational Mechanics and Analysis \cite{Willems1972-1}, \cite{Willems1972-2}.  In these  seminal papers he introduced the notion of a dissipative system. During the same period, he also made fundamental contributions to the subject of optimal control, in particular to linear quadratic problems with indefinite cost, and the associated algebraic Riccati equation \cite{Willems1971}. Indeed, the most appealing framework for studying the Riccati equation is that of dissipative systems, because there the Riccati equation emerges in a natural way by reformulating the dissipation inequality (which expresses the fact that the system under consideration is dissipative) as a so-called linear matrix inequality (LMI). Together, \cite{Willems1972-1}, \cite{Willems1972-2} and \cite{Willems1971} are generally considered to provide the main concepts and analysis tools in many areas of linear and nonlinear systems and control, ranging from stability theory, linear quadratic optimal control and stochastic realization theory, to network synthesis, differential games and robust control.   

In the early 1980s, Jan Willems became conscious of the limitations of input-output thinking as the framework for systems and control theory. This led him to develop the behavioral approach, in which a dynamical system is simply viewed as a family of trajectories.  The fourth author of this paper has had the opportunity to collaborate with Jan Willems on dissipativity in the context of the behavioral approach to linear systems. A central concept in that theory is the notion of quadratic differential form (QDF), introduced  in 1998 in \cite{Willems1998}.  In the behavioral theory of dissipativity, supply rates and storage functions are represented by QDFs. Using this framework, in \cite{Trentelman1997} it was shown that every storage function of a given dissipative system can be written as a quadratic function of the state of that system. Also, in \cite{Willems2002} and \cite{Trentelman2002},  the well-known $H_{\infty}$ control problem for linear input-state-output systems \cite{Doyle1989} was reformulated as a behavioral control problem of finding a suitable dissipative controlled behavior. The famous conditions from \cite{Doyle1989} (requiring the existence of two solutions of Riccati equations that satisfy a coupling condition) were generalized in \cite{Willems2002}, requiring the existence of storage functions (as QDFs) satisfying a coupling condition.  

In addition, it was shown in \cite{Willems1998} that computations on QDFs can be represented as formal operations on two-variable polynomial matrices. Using this framework, this led the third and fourth authors of the present paper to apply behavioral dissipativity theory to the problems of $J$-spectral factorization \cite{Trentelman1999} and the existence of sign-definite solutions to the algebraic Riccati equation \cite{Trentelman2001}. Also the second author has employed dissipativity theory in the contexts of uncontrollable systems \cite{Camlibel2003} and discontinuous systems \cite{Camlibel2002,Frasca2010,Camlibel2016}.

In the present paper, we study dissipativity of linear finite-dimensional input-state-output systems from a data-driven perspective. It is well-known that for a given input-state-output system with given supply rate, one can test dissipativity by checking the feasibility of a linear matrix inequality involving the system matrices. In this paper, we assume that the system dynamics is unknown, in the sense that we do know the state space dimension and the input and output dimensions, but the system matrices are not known. In this situation, the question arises whether we can verify dissipativity using measured system trajectories, instead of a system model.  

Recently, the problem of inferring dissipativity properties from data has received considerable attention. In \cite{Romer2017}, the set of supply rates with a given structure with respect to which a (not necessarily linear) system is dissipative was computed on the basis of a finite number of its input-output trajectories. In \cite{Romer2017b} an iterative procedure was illustrated to compute the input feedforward passivity index and the shortage of passivity for discrete-time linear systems. The most relevant references for the problem studied in this paper are \cite{Maupong2017a,Romer2019,Koch2020b}. In \cite{Maupong2017a}, the notion of (finite-horizon) L-dissipativity was introduced and also studied in \cite{Romer2019}. A discrete-time system is $L$-dissipative if the average of the supply rate over the interval $[0,L]$ is nonnegative for all system trajectories. This is a sufficient condition for dissipativity, equivalent to a matrix inequality set up from the supply rate and a system trajectory over $[0,L]$. In both these contributions, a crucial assumption is that the input trajectory is persistently exciting of a sufficiently high order (see \cite{Willems2005} and \cite{vanWaarde2020c}). This property of the input sequence can be shown to imply that the data-generating system is uniquely identifiable from the data.  

In this paper we adopt the more classical notion of dissipativity for linear systems, rather than $L$-dissipativity. We consider a setup similar to that of \cite{Koch2020b}. In that paper, sufficient data-based conditions were given for dissipativity. The main difference  between our results and those in \cite{Koch2020b} is that we provide necessary and sufficient conditions for dissipativity based on data, for noiseless and noisy data. An additional (but smaller) difference is that in our setting, also the output system matrices are unknown.

Our approach involves bounding the noise by a quadratic matrix inequality, which implies that also the unknown system parameters satisfy a quadratic matrix inequality. Our goal is to ascertain dissipativity of \emph{all} systems satisfying this inequality. The method thus fits in the robust control literature, where quadratic uncertainty descriptions have been studied in detail. We mention contributions to integral quadratic constraints \cite{Megretski1997}, the quadratic separator \cite{Iwasaki1998}, and the full block S-procedure \cite{Scherer1997,Scherer2001}. At a high level, our approach differs from classical robust control in that we provide direct mappings from data to storage functions. At a more technical level, we make use of a new matrix S-lemma that was recently established in \cite{vanWaarde2020d} as a generalization of the famous S-lemma by V.~A.~Yakubovich \cite{Yakubovich1977}. We refer to \cite{vanWaarde2020d} for a more detailed comparison between this approach and the robust control literature.

Specifically, our contributions are the following. First, we prove that dissipativity of an unknown linear system can only be ascertained on the basis of the given data if a matrix constructed from measured states and inputs has full rank. In the noiseless data case, this implies that one can only verify dissipativity from data if the data-generating system is the only one that explains the data, in other words, if the true system is identifiable from the data. In this case, dissipativity of the unknown system can be ascertained by checking the feasibility of a given data-based linear matrix inequality.  In the noisy data case, it turns out that one does not need identifiability. In order to check dissipativity in this case, we combine the matrix S-lemma with a basic dualization lemma to provide a data-driven test for dissipativity.

\subsection*{Notation}

The \emph{inertia} of a symmetric matrix $S$ is denoted by $\In(S)=(\rho_-,\rho_0,\rho_+)$ where $\rho_-$, $\rho_0$, and $\rho_+$ respectively denote the number of negative, zero, and positive eigenvalues of $S$. The \emph{interior} of a set $V$ is denoted by $\inter(V)$.

\section{Dissipativity of linear systems}
\label{sec:dis}
Consider a linear discrete-time input/state/output system
\bse\label{e:lin-sys}
\begin{align}
\bm x(t+1)&=A \bm x(t)+B \bm u(t) \\
\bm y(t)&=C \bm x(t)+D \bm u(t)
\end{align}
\ese
where $A\in\R^{n\times n}$, $B\in\R^{n\times m}$, $C\in\R^{p\times n}$, and $D\in\R^{p\times m}$. 

Let $S=S^\top  \in\R^{(m+p)\times(m+p)}$. The system \eqref{e:lin-sys} is said to be \emph{dissipative\/} with respect to the {\em supply rate}  
\beq\label{e:supply}
s(u,y)=\bbm u\\y\ebm^\top S \bbm u\\y\ebm
\eeq
if there exists $P\in\R^{n\times n}$ with $P=P^\top   \geq 0$ such that the {\em dissipation inequality\/}
\begin{equation}
\label{dispineq}
\bm x(t)^\top P \bm x(t) + s\big(\bm u(t),\bm y(t) \big) \geq  \bm x(t+1)^\top P \bm x(t+1)
\end{equation} 
holds for all $t \geq 0$ and for all  trajectories $(\bm u,\bm x,\bm y): \mathbb{N}\rightarrow \mathbb{R}^{m+n+p}$ of \eqref{e:lin-sys}.

It follows from \eqref{dispineq} that dissipativity with respect to the supply rate \eqref{e:supply} is equivalent with the feasibility of the linear matrix inequalities $P=P^\top \geq 0$ and 
\setlength\arraycolsep{2pt}
\beq\label{eq:KY_- P}
\bbm
I & 0 \\A & B
\ebm^\top
\bbm
P & 0\\0 & -P
\ebm
\bbm
I & 0 \\A & B
\ebm+
\bbm
0 & I\\C & D
\ebm^\top
S
\bbm
0 & I\\C & D
\ebm
\geq 0.
\eeq

\section{Problem formulation}
\label{sec:inf}
Consider the linear discrete-time input/state/output system
\bse\label{e:tru-sys}
\begin{align}
\bm x(t+1)&=A_{s} \bm x(t)+B_{s} \bm u(t)+\bm w(t)\\
\bm y(t)&=C_{s} \bm x(t)+D_{s} \bm u(t)+\bm z(t)\end{align}
\ese
where $(\bm u,\bm x,\bm y)\in\R^{m+n+p}$ are the input, state and output, and $(\bm w,\bm z)\in\R^{n+p}$ are noise terms. Throughout the paper, we assume that the ``true" system matrices $(A_{s}, B_{s},C_{s},D_{s})$ and the noise $(\bm w,\bm z)$ are {\em unknown\/}. What is known instead are a finite number of input/state/output measurements of \eqref{e:tru-sys}, which we collect in the matrices
\bse\label{eq:datamatrix}
\begin{align*}
U_- &:=\begin{bmatrix}
u(0) & u(1) & \cdots & u(T-1)
\end{bmatrix}\\
X&:=\begin{bmatrix}
x(0)& x(1)&\cdots& x(T) 
\end{bmatrix}\\
Y_- &:= \begin{bmatrix}
y(0) & y(1) & \cdots & y(T-1)
\end{bmatrix}.
\end{align*}
\ese
We will also make use of the auxiliary matrices
\begin{align*}
X_-&:=\begin{bmatrix}
x(0)& x(1)&\cdots& x(T-1)
\end{bmatrix}\\
X_+&:=\begin{bmatrix}
x(1) & x(2) & \cdots & x(T)
\end{bmatrix}.
\end{align*}
The goal of this paper is to infer dissipativity properties of the true system from the data $(U_- ,X,Y_- )$.

\newcommand{\sig}{\Sigma_{\calN}^{U_- ,X,Y_- }}
\newcommand{\sigzero}{\Sigma_{\calN_0}^{U_- ,X,Y_- }}
\newcommand{\sigone}{\Sigma_{\calN_2}^{U_- ,X,Y_- }}
\newcommand{\sigtwo}{\Sigma_{\calN_1}^{U_- ,X,Y_- }}

 We define
$$
\Sigma^\mathcal{N}=\left\{\left(A,B,C,D\right) \mid \begin{bmatrix} X_+\\Y_-  \end{bmatrix}- \begin{bmatrix} A&B\\
C&D\end{bmatrix}\begin{bmatrix}X_-\\ U_-  \end{bmatrix}\in\calN \right\},
$$
where $\calN\subseteq\R^{(n+p)\times T}$ is a set defining a \emph{noise model} to be specified below. We assume that
\beq\label{e:true in exp}
(A_{s},B_{s},C_{s},D_{s})\in \Sigma^\mathcal{N}.
\eeq
In the sequel, we will consider three types of noise models. The first one will capture noise-free situations in which the measurements $(U_- ,X,Y_- )$ are exact: 
\begin{equation}\label{eq:N0}
\calN_0:=\zset.
\end{equation}

The second noise model is defined by
\begin{equation}\label{eq:N1}
\calN_1:= \left\{ V\in\R^{(n+p)\times T} \mid \bbm
I \\ V^\top
\ebm
^\top
\!\!
\bbm
\Phi_{11} & \Phi_{12}\\
\Phi_{12}^\top & \Phi_{22}
\ebm\!\!
\bbm
I \\ V^\top 
\ebm
\geq 0 \right\},
\end{equation}
where $\Phi_{11}=\Phi_{11}^\top\in\R^{(n+p)\times(n+p)}$, $\Phi_{12}\in\R^{(n+p)\times T}$, and $\Phi_{22}=\Phi_{22}^\top\in\R^{T\times T}$ are known matrices. This noise model was studied before \cite{vanWaarde2020d} in the context of data-driven quadratic stabilization and $H_2$ and $H_{\infty}$ control. In order to be able to discuss some special cases of the noise model \eqref{eq:N1}, we label the columns of $V$ as $\begin{bmatrix} v(0) & v(1) & \cdots & v(T-1) \end{bmatrix}$. In the special case $\Phi_{12} = 0$ and $\Phi_{22} = -I$, the bound in \eqref{eq:N1} reduces to 
\begin{equation}
    \label{redasnoise}
    V V^\top = \sum_{t=0}^{T-1} v(t) v(t)^\top \leq \Phi_{11}.
\end{equation}
This inequality can be interpreted as a type of energy bound on the noise. If $\bmv$ is a random variable, the \emph{sample covariance matrix} of $v(0),v(1),\dots,v(T-1)$ is given by 
$$
\frac{1}{T-1} V (I-\frac{1}{T}J) V^\top,
$$
where $J$ is the matrix of ones. Thus, the noise model \eqref{eq:N1} can also capture known bounds on the sample covariance by the choices $\Phi_{12} = 0$ and $\Phi_{22} = -\frac{1}{T-1} (I-\frac{1}{T}J)$. We emphasize, however, that we do not make any assumptions on the statistics of the noise and work with the general model \eqref{eq:N1} instead.

Finally, we remark that norm bounds on the individual noise samples $v(t)$ also give rise to bounds of the form \eqref{redasnoise}, although this generally leads to some conservatism. Indeed, note that $\norm{v(t)}^2 \leq \epsilon$ implies that $v(t) v(t)^\top \leq \epsilon I$ for all $t$. As such, the bound \eqref{redasnoise} is satisfied for $\Phi_{11} = T \epsilon I$.

The third noise model that we will consider is defined  by
\begin{equation}\label{eq:N2}
\calN_2:=\left\{ V\in\R^{(n+p)\times T\!} \mid 
\bbm
I \\V
\ebm^\top\!\!
\bbm
\Theta_{11} & \Theta_{12}\\
\Theta_{12}^\top & \Theta_{22}
\ebm\!\!
\bbm
I \\ V 
\ebm \geq 0
\right\}
\end{equation}
where $\Theta_{11}=\Theta_{11}^\top\in\R^{T\times T}$, $\Theta_{12}\in\R^{T\times(n+p)}$, and $\Theta_{22}=\Theta_{22}^\top\in\R^{(n+p)\times(n+p)}$ are known matrices. This noise model was studied before in \cite{Berberich2019c} in the context of data-driven state feedback control. It turns out that under mild assumptions, it is possible to convert noise model $\calN_1$ to $\calN_2$ and vice versa. The interested reader is referred to Corollary~\ref{c:eqnoisemodels}.

We now define the  property of \emph{informativity for dissipativity}, which is the main concept studied in this paper. This definition is inpired by \cite{vanWaarde2020}, and we refer to that paper for a general treatment of data informativity and an application to data-driven controllability analysis, stabilization and optimal control. 

\begin{definition}\label{def:dd diss}
Let a noise model $\calN$ be given. The data $(U_-,X,Y_-)$ are \emph{informative for dissipativity\/} with respect to the supply rate \eqref{e:supply} if there exists a matrix $P=P^\top\geq 0$ such that the LMI \eqref{eq:KY_- P} holds for every system $(A,B,C,D)\in\Sigma^\mathcal{N}$. 
\end{definition}

The rationale behind Definition~\ref{def:dd diss} is as follows: on the basis of the given data we are unable to distinguish between the systems in $\Sigma^\mathcal{N}$ in the sense that any of these systems could have generated the data. Nonetheless, if \emph{all} of these systems are dissipative, then we can also conclude that the \emph{true} data-generating system is dissipative. Note that we restrict our attention to the situation in which the systems in $\Sigma^\mathcal{N}$ are dissipative with a \emph{common} storage function.

The following assumptions will be valid throughout the paper: 
\begin{enumerate}[({A}1)]
\item The matrix $S$ has inertia $\In(S)=(p,0,m)$.\label{ass:supply}
\item The sets $\calN_1$ and $\calN_2$ are bounded and have nonempty interior.\label{ass:n1 and n2}
\end{enumerate}

It is a well-known fact that a necessary condition for dissipativity of any system of the form \eqref{e:lin-sys} is that $m \leq \rho_+$, i.e., the input dimension does not exceed the positive signature of $S$. Assumption (A1) requires that the input dimension is equal to this positive signature, and in addition that the matrix $S$ is nonsingular. This assumption is satisfied, for example, for the positive-real and bounded-real case \cite{Scherer1999}. Indeed, in the positive-real case we have that $m = p$ and
$$
S = \begin{bmatrix}
0 & I_m \\ I_m & 0
\end{bmatrix},
$$
so that $\In(S) = (m,0,m)$. In the bounded-real case we have 
$$
S = \begin{bmatrix}
\gamma^2 I_m & 0 \\ 0 & -I_p
\end{bmatrix}
$$
for $\gamma > 0$, which implies that $\In(S)=(p,0,m)$. Assumption (A1) also turns out to be instrumental in computing storage functions of the ``dual" system from those of the primal one, see Proposition~\ref{p:diss dual}. Assumption (A2) can be verified straightforwardly by assessing certain definiteness properties of the matrices $\Phi$ and $\Theta$ in \eqref{eq:N1} and \eqref{eq:N2}; for more details see ``\nameref{sidebar:(A2)}".

The main contribution of this paper is to provide necessary and sufficient conditions for data informativity for the noise models $\calN_0$, $\calN_1$, and $\calN_2$.

\section{Main results}
\label{sec:res}

\subsection{A necessary condition for informativity}
We begin with a necessary condition for informativity, that applies to all three noise models.
\bthe\label{t:gen}
Let a noise model $\calN$ be given. If the data $(U_- ,X,Y_- )$ are informative for dissipativity with respect to the supply rate \eqref{e:supply}, then
\beq\label{e:full row rank}
\rank \bbm X_-\\U_- \ebm = n + m.
\eeq
\ethe

Essentially, Theorem~\ref{t:gen} and the rank condition \eqref{e:full row rank} formalize the intuition that dissipativity can only be assessed from data that are sufficiently rich. In the noise-free setting (involving model $\mathcal{N}_0$), the rank condition \eqref{e:full row rank} implies that the system matrices $A_s,B_s,C_s$ and $D_s$ can be uniquely identified from the $(U_-,X,Y_-)$-data. In this setting, the interpretation of Theorem~\ref{t:gen} is that dissipativity can \emph{only} be verified from data that are rich enough to uniquely identify the underlying data-generating system.  

\begin{proof}
Suppose that \eqref{e:full row rank} does not hold. Then, there exist $\xi\in\R^n$ and $\eta\in\R^m$ such that $\xi^\top\xi+\eta^\top\eta=1$ and
\beq\label{e:xi eta in left kernel}
\bbm\xi^\top &\eta^\top \ebm\begin{bmatrix}X_-\\ U_-  \end{bmatrix}=0.
\eeq
The set $\Gamma= \{u \mid \exists\,y \text{ such that }s(u,y)<0\}$ has nonempty interior since there exists $(\hatu,\haty)$ with $s(\hatu,\haty)<0$ due to Assumption~(A1). We claim that there exist $x\in\R^n$ and $u\in\Gamma$ such that
\beq\label{e:xi eta x u}
\xi^\top x+\eta^\top u=1.
\eeq
Indeed, if $\xi \neq 0$, then one can construct $x$ and $u$ by selecting $u \in \Gamma$ arbitrarily, and by defining $x := \frac{1-\eta^\top u}{\xi^\top \xi} \xi$. If $\xi = 0$ then $x \in \mathbb{R}^n$ can be selected arbitrarily. In this case, we can choose $u$ as follows. Since $\Gamma$ has nonempty interior, there exists $\bar{u} \in \Gamma$ such that $\eta^\top \bar{u} \neq 0$. Note that $\alpha \bar{u} \in \Gamma$ for all nonzero $\alpha \in \mathbb{R}$. As such, there exists an $\alpha \in \mathbb{R}$ such that $u := \alpha \bar{u} \in \Gamma$ and $\eta^\top u = 1$. For this $u$, we obtain \eqref{e:xi eta x u} which proves our claim. 

Since $u\in\Gamma$, there  exists $y$  such that $s(u,y)<0$. Let $(A_0,B_0,C_0,D_0)\in\Sigma^\mathcal{N}$. Define
\beq\label{e:zeta theta}
\zeta:=x-A_0x-B_0u \qand \theta:=y-C_0x-D_0u,
\eeq
and 
$$
\bbm
A & B\\C& D
\ebm
:=
\bbm
A_0 & B_0\\C_0& D_0
\ebm
+\bbm
\zeta\\\theta
\ebm
\bbm \xi^\top& \eta^\top\ebm.
$$
It follows from \eqref{e:xi eta in left kernel} that $(A,B,C,D)\in\Sigma^\mathcal{N}$. Since the data are informative for dissipativity with respect to the supply rate \eqref{e:supply}, there must exist $P=P^\top \geq 0$ such that 
\beq\label{eq:KY_- P2}
\bbm
I & 0 \\A & B
\ebm^\top
\bbm
P & 0\\0 & -P
\ebm
\bbm
I & 0 \\A & B
\ebm+
\bbm
0 & I\\C & D
\ebm^\top
S
\bbm
0 & I\\C & D
\ebm
\geq 0.
\eeq
Note that
$$
\bbm
I & 0 \\A & B
\ebm
\bbm
x\\u
\ebm=\bbm x\\ x\ebm\qand \bbm
0 & I\\C & D
\ebm
\bbm
x\\u
\ebm=\bbm u\\ y\ebm ,
$$
due to \eqref{e:xi eta x u} and \eqref{e:zeta theta}. Therefore, the following inequality holds: 
\begin{align*}
&\bbm x\\u \ebm^\top\!\!\!
\left(\bbm
I & 0 \\A & B
\ebm^\top
\bbm
P & 0\\0 & -P
\ebm
\bbm
I & 0 \\A & B
\ebm+
\bbm
0 & I\\C & D
\ebm^\top
S
\bbm
0 & I\\C & D
\ebm\right)\!\!
\bbm x\\u \ebm\\
&=\bbm x\\x \ebm^\top
\bbm
P & 0\\0 & -P
\ebm
\bbm x\\x \ebm
+
\bbm u\\y \ebm^\top
S
\bbm u\\y \ebm=s(u,y)<0 \; .
\end{align*}
However, this contradicts  \eqref{eq:KY_- P2}. Consequently, \eqref{e:full row rank} holds.
\end{proof}

\subsection{Informativity and noiseless data}

We now give a characterization of informativity for dissipativity for the noiseless case. We note that the condition of Theorem~\ref{t:exact} has appeared in a similar setting in \cite[Thm. 4]{Koch2020} and \cite[Thm. 3]{Koch2020b}, where an ``if"-statement was proven. Here we prove that these conditions are necessary and sufficient by leveraging Theorem~\ref{t:gen}.
\begin{theorem}\label{t:exact}
Consider the noise model $\calN_0$. The data $(U_- ,X,Y_- )$ are informative for dissipativity with respect to the supply rate \eqref{e:supply} if and only if 
\beq\label{e:exact cond1}
\rank \bbm X_-\\U_- \ebm = n+m
\eeq
and there exists $P=P^\top \geq0$ such that
\beq\label{e:exact cond2}
\bbm
X_-\\X_+
\ebm^\top
\bbm
P & 0\\0 & -P
\ebm
\bbm
X_-\\X_+
\ebm+
\bbm
U_- \\Y_- 
\ebm^\top
S
\bbm
U_- \\Y_- 
\ebm
\geq 0.
\eeq
\end{theorem}
\medskip
\begin{proof}
To prove the ``if" part, note that \eqref{e:exact cond1} implies that $\Sigma^{\mathcal{N}_0}$ is a singleton. It follows from \eqref{e:true in exp} that 
$$
\Sigma^{\mathcal{N}_0}=\pset{(A_{s},B_{s},C_{s},D_{s})}
$$
and hence 
$$
\bbm X_+\\Y_- \ebm=\bbm A_{s} & B_{s}\\C_{s}& D_{s}\ebm
\bbm X_-\\U_- \ebm.
$$
Define 
\begin{align*}
L :=
\bbm
I & 0 \\A_{s} & B_{s}
\ebm^\top
\!
\bbm
P & 0\\0 & -P
\ebm
\!
\bbm
I & 0 \\A_{s} & B_{s}
\ebm + \bbm
0 & I\\C_{s} & D_{s}
\ebm^\top
\!S
\bbm
0 & I\\C_{s} & D_{s}
\ebm.
\end{align*}
Then \eqref{e:exact cond2} implies 
\beq\label{e:exact lmi with x-u}
\bbm
X_-\\U_- 
\ebm^\top
L
\bbm
X_-\\U_- 
\ebm
\geq 0.
\eeq
It follows again from \eqref{e:exact cond1} that $L\geq 0$. By \eqref{eq:KY_- P}, this means that the system $(A_{s},B_{s},C_{s},D_{s})$ is dissipative with respect to the supply rate \eqref{e:supply}.\\

To prove the ``only if" part, note that it follows from Theorem~\ref{t:gen} that \eqref{e:exact cond1} holds. Hence we have
$$
\Sigma^{\mathcal{N}_0}=\pset{(A_{s},B_{s},C_{s},D_{s})}.
$$ 
Since the data are informative for dissipativity for the given $\mathcal{N}_0$, there exists $P=P^\top\geq 0$ such that $L \geq 0$. By post- and pre-multiplying this expression by $\bbm X_-\\U_- \ebm$ and its transpose, we conclude that  
\eqref{e:exact cond2} holds.
\end{proof}

Theorem~\ref{t:exact} provides a data-based condition for dissipativity in terms of a linear matrix inequality. Linear matrix inequalities can be solved using standard software packages. We note, however, that such solvers are known to be unreliable for LMI's which define
feasible sets without interior points. As such, from a numerical point of view it is desirable that there exists a positive definite $P$ such that
left-hand side of \eqref{e:exact cond2} is positive definite.

\begin{remark}\label{rem:exact}
Condition \eqref{e:exact cond1} implies that even if the  state is measured (a more advantageous situation than knowing only the input-output data, as is typically assumed in data-driven applications), it is \emph{only} possible to ascertain dissipativity from noise-free data if the plant is \emph{uniquely} identifiable, i.e., if $|\Sigma^{\mathcal{N}_0}| = 1$. Consequently, in the noise-free setting, methods for determining dissipativity directly from data are conceptually equivalent with indirect ones consisting of a system identification stage, followed by a second one involving a  check on the solvability of an LMI (condition \eqref{eq:KY_- P}).
\end{remark}

\subsection{Informativity and noisy data}
We first consider the noise model $\mathcal{N}_1$ defined in \eqref{eq:N1}. Define 
\beq
N_1:=\None.  
\eeq
Note that $(A,B,C,D)\in \Sigma^{\mathcal{N}_1}$ if and only if
\beq\label{e:char N2 model}
\systwo^\top\!\! N_1 \systwo\geq 0.
\eeq
Partition 
$$
S=\bbm F & G\\G^\top & H\ebm,
$$
where  $F\in\R^{m\times m}$, $G\in\R^{m\times p}$, $H\in\R^{p\times p}$, and define
\begin{equation*}
M_1:=\bbm
P & 0 & 0 & 0\\
0 & F & 0 & G\\
0 & 0 & -P & 0\\
0 & G^\top & 0 & H
\ebm.
\end{equation*}
With this notation in place, the problem of characterizing informativity for dissipativity is equivalent to finding conditions  under which the inequality
\begin{equation}
\label{eqM1}
\sysone^\top\!\! M_1 \sysone \geq 0
\end{equation} 
holds for all $(A,B,C,D)$ satisfying
\begin{equation}
\label{eqN1}
\systwo^\top\!\! N_1 \systwo \geq 0 \; .
\end{equation}

Our strategy to solve this problem is to invoke the so called matrix S-lemma \cite{vanWaarde2020d}. The matrix $S$-lemma is a powerful generalization to matrix variables of the classical $S$-lemma (also called $S$-procedure) developed in the seventies of the previous century by V. A. Yakubovich \cite{Yakubovich1977}. We recall this result below.

\begin{proposition}[Matrix S-lemma]
\label{matSlemma}
Let $M,N \in \mathbb{R}^{(q+r) \times (q+r)}$ be symmetric matrices. Assume that there exists a matrix $\bar{Z} \in \mathbb{R}^{r \times q}$ such that
\begin{equation}
    \label{matSlater}
\begin{bmatrix} I \\ \bar{Z} \end{bmatrix}^\top\!\!\! N \begin{bmatrix} I \\ \bar{Z} \end{bmatrix} > 0.
\end{equation}
Then we have that
    $$
    \begin{bmatrix} I \\ Z \end{bmatrix}^\top\!\!\! M \begin{bmatrix} I \\ Z \end{bmatrix} \geq 0 \:\: \text{ for all } \,  Z \in \mathbb{R}^{r \times q} \text{ satisfying } \begin{bmatrix} I \\ Z \end{bmatrix}^\top\!\!\! N \begin{bmatrix} I \\ Z \end{bmatrix} \geq 0
    $$
    if and only if there exists a scalar $\alpha \!\geq 0$ such that $M - \alpha N \! \geq 0$.
\end{proposition}

Before we can apply Proposition \ref{matSlemma} we note that the inequality \eqref{eqM1} is in terms of $(A,B,C,D)$ while the inequality \eqref{eqN1} is in terms of the \emph{transposed} matrices $(A^\top,C^\top,B^\top,D^\top)$. Therefore, we will need an additional dualization result that we formulate in the following proposition. 

\begin{proposition}\label{p:diss dual}
Consider a real $n \times n$ matrix $P=P^\top > 0$ and a matrix
$$
\begin{bmatrix}
A&B\\
C&D
\end{bmatrix}\in\R^{(n+p)\times (n+m)}.
$$
Suppose that $S=S^\top\in\R^{(m+p)\times (m+p)}$ satisfies Assumption~(A1). Define
\begin{equation}
\label{Shat}
 \hat{S}:=\begin{bmatrix}0&-I_p\\
 I_m&0 \end{bmatrix} S\inv \begin{bmatrix} 0&-I_m\\I_p&0 \end{bmatrix}.
\end{equation}
Then we have that
\begin{equation}
\label{eqL1}
\begin{bmatrix}
I&0\\
A&B
\end{bmatrix}^\top\!\!\begin{bmatrix}
P&0\\
0&-P
\end{bmatrix}\begin{bmatrix}
I&0\\
A&B
\end{bmatrix}+\begin{bmatrix}
0&I\\
C&D
\end{bmatrix}^\top\!\! S \begin{bmatrix}
0&I\\
C&D
\end{bmatrix}\geq 0
\end{equation}
if and only if 
\begin{align}\label{eq:L2}
\begin{bmatrix}
I&0\\
A^\top&C^\top
\end{bmatrix}^\top\!\!\begin{bmatrix}
P^{-1}&0\\
0&-P^{-1}
\end{bmatrix}
\begin{bmatrix}
I&0\\
A^\top&C^\top
\end{bmatrix}+\begin{bmatrix}
0&I\\
B^\top&D^\top
\end{bmatrix}^\top\!\! \hat{S} \begin{bmatrix}
0&I\\
B^\top&D^\top
\end{bmatrix} \geq 0.
\end{align}
\end{proposition}
\begin{proof}
Partition the matrix $S$ as 
$$
S = \begin{bmatrix}
S_{11} & S_{12} \\ S_{12}^\top & S_{22}
\end{bmatrix},
$$
where $S_{11} \in \mathbb{R}^{m \times m}$, $S_{12} \in \mathbb{R}^{m \times p}$ and $S_{22} \in \mathbb{R}^{p \times p}$. By positive definiteness of $P$ and Assumption (A1), the matrix 
\begin{equation}
\label{psidiss}
\Psi := \begin{bmatrix}
P & 0 & 0 & 0 \\ 0 & S_{11} & 0 & S_{12} \\ 
0 & 0 & -P & 0 \\ 0 & S_{12}^\top & 0 & S_{22}
\end{bmatrix}
\end{equation} 
satisfies $\In(\Psi) = (n+p,0,n+m)$. The proposition now readily follows by applying Lemma~\ref{l:dualizationresult} to the matrix $\Psi$ in \eqref{psidiss}.
\end{proof}

Proposition~\ref{p:diss dual} can be interpreted as saying that the system defined by the quadruple $(A,B,C,D)$ is dissipative with respect to the supply rate $S$, with storage function $P$ if and only if the dual system $(A^\top,C^\top,B^\top,D^\top)$ is dissipative with respect to the supply rate $\hat{S}$, with storage function $P^{-1}$.
A behavioral analogue of this result was obtained in \cite{Willems2002}, Proposition 12. 

By combining Propositions \ref{matSlemma} and \ref{p:diss dual}, we now arrive at the following characterization of informativity for dissipativity, given the noise model $\mathcal{N}_1$. 

\bthe\label{t:noise 1}
Suppose that there exists $V\in\R^{(n+m)\times(n+p)}$ such that
\beq\label{e:slater for n2}
\bbm I\\V\ebm^\top\!\! N_1 \bbm I\\V\ebm> 0.
\eeq
Partition 
$$
\begin{bmatrix}
\hatF&\hatG\\
\hatG^\top  &\hatH
\end{bmatrix}:=-S^{-1},
$$
where $\hatF=\hatF^\top \in\R^{m\times m}$, $\hatG\in\R^{m\times p}$, and $\hatH=\hatH^\top \in\R^{p\times p}$. 
Given the noise model $\calN_1$, the data $(U_- ,X,Y_- )$ are informative for dissipativity with respect to the supply rate \eqref{e:supply} if and only if there exist a real $n \times n$ matrix $Q=Q^\top>0$ and a scalar $\alpha\geq 0$ such that 
\begin{equation}
\begin{bmatrix}
    Q & 0 & 0 & 0 \\
    0 & \hatH & 0 &-\hatG^\top\\
    0 & 0 & -Q & 0 \\
    0 & -\hatG & 0 & \hatF
    \end{bmatrix} - \alpha 
    \None
     \geq 0. \label{LMI2}\tag{$\text{LMI}$}
\end{equation}
\ethe

We note that a sufficient condition for data-driven dissipativity with a common storage function was given in \cite[Thm. 4]{Koch2020b}. The attractive feature of Theorem~\ref{t:noise 1} is that it provides a necessary and sufficient condition, by making use of the matrix S-lemma.

We will prove Theorem~\ref{t:noise 1} by means of the following auxiliary lemma. This lemma shows that if all systems in $\Sigma^{\mathcal{N}_1}$ are dissipative with common storage function $P = P^\top \geq 0$, then $P$ is necessarily \emph{positive definite}. We note that conditions under which all storage functions are positive definite have been studied before in \cite[Lem. 1]{Hill1976}, even for nonlinear systems. In that paper, certain minimality conditions were imposed as well as a signature condition on the supply rate. Here, we do not assume minimality but we conclude that all storage functions are positive definite by using Assumption (A1) and an argument related to the noise model.
\begin{lemma}
\label{l:P>0}
Suppose that there exists a matrix $V$ such that \eqref{e:slater for n2} holds. If $P=P^\top\geq 0$ satisfies the dissipation inequality \eqref{eq:KY_- P} for all $(A,B,C,D) \in \Sigma^{\mathcal{N}_1}$ then $P > 0$. 
\end{lemma}
\begin{proof} 
Let $\xi\in\ker P$. It follows from \eqref{eq:KY_- P} that
$$
\bbm
\alpha\\\eta
\ebm^\top\!\!\!
\left(
-\bbm
\xi^\top\!\! A^\top \\ B^\top
\ebm
P
\bbm
A\xi & B
\ebm
+
\bbm
0 & I\\C\xi\ & D
\ebm^\top\!\!\!
S
\bbm
0 & I\\C\xi & D
\ebm
\right)
\bbm
\alpha\\\eta
\ebm \geq 0 
$$
for all $\alpha\in\R$, $\eta\in\R^m$, and $(A,B,C,D)\in\Sigma^{\mathcal{N}_1}$. This implies that
$$
\bbm
0 & I\\C\xi\ & D
\ebm^\top\!\!\!
S
\bbm
0 & I\\C\xi & D
\ebm\geq 0
$$
for every $(A,B,C,D)\in\Sigma^{\mathcal{N}_1}$. It follows from \cite[Theorem 3.1]{Dancis1986} that
$$
\dim\left(\ker \bbm
0 & I\\C\xi & D
\ebm\right)\geq 1.
$$
Therefore, $C\xi=0$ for every $(A,B,C,D)\in\Sigma^{\mathcal{N}_1}$. By hypothesis, the set $\Sigma^{\mathcal{N}_1}$ has nonempty interior. Consequently, we can conclude that $\xi=0$ and hence $P>0$. 
\end{proof}

\emph{Proof of Thm.~\ref{t:noise 1}:} To prove the ``if" statement, let $(A,B,C,D) \in \Sigma^{\mathcal{N}_1}$. We multiply \eqref{LMI2} from right and left by 
$$
\begin{bmatrix}
I & 0 \\ 
0 & I \\
A^\top & C^\top \\
B^\top & D^\top
\end{bmatrix}
$$
and its transpose. By the assumption on the noise (see Equation \eqref{eq:N1}), this leads to 
$$
\begin{bmatrix}
I&0\\
A^\top&C^\top
\end{bmatrix}^\top\!\!\begin{bmatrix}
Q&0\\
0&-Q
\end{bmatrix}\!\!
\begin{bmatrix}
I&0\\
A^\top&C^\top
\end{bmatrix}\notag+\begin{bmatrix}
0&I\\
B^\top&D^\top
\end{bmatrix}^\top\!\! \hat{S} \! \begin{bmatrix}
0&I\\
B^\top&D^\top
\end{bmatrix} \! \geq \! 0,
$$ 
where $\hat{S}$ is related to $S$ via \eqref{Shat}. Finally, by Proposition \ref{p:diss dual} we conclude that \eqref{eqL1} holds for $P = Q^{-1}$. That is, the data $(U_-,X,Y_-)$ are informative for dissipativity with respect to the supply rate \eqref{e:supply}.

To prove the ``only if" part, suppose that the data $(U_-,X,Y_-)$ are informative for dissipativity, equivalently, there exists a matrix $P = P^\top \geq 0$ such that \eqref{eq:KY_- P} holds for all $(A,B,C,D) \in \Sigma^{\mathcal{N}_1}$. Note that $P>0$ by Lemma~\ref{l:P>0}. Now, by Proposition \ref{p:diss dual} it follows that \eqref{eq:L2} holds for all $(A,B,C,D) \in \Sigma^{\mathcal{N}_1}$. We define $Q:=P^{-1}$. By rearranging terms in \eqref{eq:L2} we see that
$$
\begin{bmatrix}
I & 0 \\ 
0 & I \\
A^\top & C^\top \\
B^\top & D^\top
\end{bmatrix}^\top 
\begin{bmatrix}
    Q & 0 & 0 & 0 \\
    0 & \hatH & 0 &-\hatG^\top\\
    0 & 0 & -Q & 0 \\
    0 & -\hatG & 0 & \hatF
    \end{bmatrix}
\begin{bmatrix}
I & 0 \\ 
0 & I \\
A^\top & C^\top \\
B^\top & D^\top
\end{bmatrix} \geq 0
$$
holds for all $(A,B,C,D) \in \Sigma^{\mathcal{N}_1}$, i.e., for all $(A,B,C,D)$ satisfying \eqref{eqN1}. Finally, by Proposition \ref{matSlemma} there exists a scalar $\alpha \geq 0$ such that \eqref{LMI2} holds. This completes the proof. \hfill $\blacksquare$

Theorem~\ref{t:noise 1} provides a tractable method for verifying informativity for dissipativity, given the noise model $\mathcal{N}_1$. The procedure involves solving the linear matrix inequality \eqref{LMI2} for $Q$ and $\alpha$. Given $Q$, a common storage function $P$ for all systems in $\Sigma^{\mathcal{N}_1}$ is also readily computable as $P = Q^{-1}$. 

\begin{remark}
It follows from Corollary~\ref{c:eqnoisemodels} that, under Assumption (A2), we can always transform the noise model $\mathcal{N}_1$ to $\mathcal{N}_2$ and vice versa. Therefore, by combining Theorem~\ref{t:noise 1} and Corollary~\ref{c:eqnoisemodels} we can also establish necessary and sufficient conditions for informativity for dissipativity given the noise model $\mathcal{N}_2$. This again results in a data-based LMI condition for dissipativity, analogous to \eqref{LMI2}. 
\end{remark}

\section{Conclusions}
\label{sec:conc}
In this paper we have provided methods to verify dissipativity properties of linear systems directly from measured data. We have considered both the case of exact data and the case that the data are corrupted by noise. In the case of exact data, we have proven that one can only ascertain dissipativity of a system from given data if the system can be uniquely identified from the data. If this is the case, dissipativity can be verified by means of a data-based linear matrix inequality. In the case of noisy data, we have combined the matrix S-lemma \cite{vanWaarde2020d} with a dualization property relating dissipativity properties of the original system with those of its dual to characterize data informativity for dissipativity. As in the noiseless case, also in this setting, dissipativity properties of the data-generating system can be ascertained if a data-based LMI is solvable.
 
Apart from conditions for verifying dissipativity based on data, we have derived a number of additional results as byproducts that are interesting in their own right. First of all, we have shown in Corollary~\ref{c:eqnoisemodels} that, under mild assumptions, the different noise models studied in \cite{Berberich2019c} and \cite{vanWaarde2020d} are actually \emph{equivalent}. Moreover, it follows from Lemma~\ref{l:P>0} that informativity for dissipativity (given the noise model $\mathcal{N}_1$) requires the common storage function to be \emph{positive definite}. This is a surprising conclusion, since the definition of dissipativity only requires positive semidefinite storage functions.

\newpage 

\bibliographystyle{IEEEtran}
\bibliography{references}

\newpage 

\section[How to verify Assumption (A2)?]{Sidebar: How to verify Assumption (A2)?}
\label{sidebar:(A2)}

Assumption (A2) can be verified using the following lemma. 

\blem\label{l:bounded with nonempty interior}
Let $\Psi_{11}=\Psi_{11}^\top\in\R^{q\times q}$, $\Psi_{12}\in\R^{q\times r}$, and $\Psi_{22}=\Psi_{22}^\top\in\R^{r\times r}$. Then, the set
$$
\calN := \left\{ R\in\R^{r\times q} \mid \bbm I\\R\ebm^\top\bbm \Psi_{11}&\Psi_{12}\\\Psi_{12}^\top&\Psi_{22}\ebm\bbm I\\R \ebm\geq 0 \right\}
$$
is bounded and has nonempty interior if and only if $\Psi_{22}<0$ and $\Psi_{11}-\Psi_{12}\Psi_{22}\inv\Psi_{12}^\top>0$.
\elem
\begin{proof}
Let 
$$
\Psi=\bbm \Psi_{11}&\Psi_{12}\\\Psi_{12}^\top&\Psi_{22}\ebm.
$$
To prove sufficiency, note that $\Psi_{22} < 0$ implies that $\mathcal{N}$ is bounded. Define $R := -\Psi_{22}^{-1}\Psi_{12}^\top$ and note that $\Psi_{11} - \Psi_{12} \Psi_{22}^{-1} \Psi_{12}^\top > 0$ implies 
$$
f(R):=\begin{bmatrix}
I \\ R
\end{bmatrix}^\top \Psi \begin{bmatrix}
I \\ R
\end{bmatrix} > 0.
$$
This means that $\mathcal{N}$ has nonempty interior.

To prove necessity, we first show that $\Psi_{22}<0$. Let $R \in \inter(\mathcal{N})$ and let $\xi \in \mathbb{R}^r$ be such that $\xi^\top \Psi_{22} \xi \geq 0$. Then $f(R+\alpha \xi\xi^\top (\Psi_{12}^\top + \Psi_{22}R)) \geq 0$ for all $\alpha \geq 0$. Since $\mathcal{N}$ is bounded, we conclude that 
\begin{equation}
\label{xiPsi}
\xi^\top (\Psi_{12}^\top + \Psi_{22}R) = 0.
\end{equation} 
Since $R \in \inter(\mathcal{N})$ is arbitrary, it follows that for all $R_1,R_2 \in \inter(\mathcal{N})$ the equality $\xi^\top \Psi_{22}(R_1 - R_2) = 0$ holds. This implies that $\xi^\top \Psi_{22} = 0$ and, by \eqref{xiPsi}, also $\xi^\top \Psi_{12}^\top = 0$. Now, observe that $f(R + \alpha \xi \xi^\top) = f(R) \geq 0$ for all $\alpha \in \mathbb{R}$ and $R \in \mathcal{N}$. By boundedness of $\mathcal{N}$, this implies that $\xi = 0$. Therefore, $\Psi_{22} < 0$. 

To prove the rest of the claim, let $\zeta \in \mathbb{R}^q$ and $\eta \in \mathbb{R}^r$ be such that 
\begin{equation}
\label{kernelPsi}
\Psi \begin{bmatrix}
\zeta \\ \eta
\end{bmatrix} = 0.
\end{equation}
If $R \in \mathcal{N}$ then 
\begin{eqnarray*}
0 &\leq& \zeta^\top f(R) \zeta = \begin{bmatrix}
\zeta \\ R\zeta 
\end{bmatrix}^\top \Psi \begin{bmatrix}
\zeta \\ R\zeta
\end{bmatrix}\\
& =& \left( \begin{bmatrix}
0 \\ R\zeta - \eta
\end{bmatrix} + \begin{bmatrix} \zeta \\ \eta \end{bmatrix} \right)^\top \Psi \left( \begin{bmatrix}
0 \\ R\zeta - \eta
\end{bmatrix}+ \begin{bmatrix} \zeta \\ \eta \end{bmatrix} \right),
\end{eqnarray*}
and thus $0 \leq (R\zeta-\eta)^\top \Psi_{22} (R\zeta-\eta)$. Since $\Psi_{22} < 0$, we conclude that $R\zeta - \eta = 0$ for all $R \in \mathcal{N}$. This implies that $(R_1 - R_2) \zeta = 0$ for all $R_1,R_2 \in \mathcal{N}$. Since the interior of $\mathcal{N}$ is nonempty, we conclude that $\zeta = 0$. Thus, \eqref{kernelPsi} leads to $\Psi_{22} \eta = 0$ and since $\Psi_{22}<0$, we conclude $\eta = 0$. Therefore, $\Psi$ is nonsingular. By Haynsworth's  inertia formula (see \cite[Fact 6.5.5]{Bernstein2011}),
\begin{equation}
\label{InertiaPsi}
\In(\Psi) = \In(\Psi_{22}) + \In(\Psi_{11} - \Psi_{12} \Psi_{22}^{-1} \Psi_{12}^\top).
\end{equation}
Let $\nu$ be the number of negative eigenvalues of $\Psi$. From $\Psi_{22} < 0$ and \eqref{InertiaPsi} we see that $\nu \geq r$. Since $\mathcal{N}$ is nonempty, it follows from \cite[Thm. 3.1]{Dancis1986} that $\nu \leq r$. Therefore, we conclude that $\nu = r$. Since $\Psi$ is nonsingular, \eqref{InertiaPsi} implies that $\Psi_{11} - \Psi_{12} \Psi_{22}^{-1} \Psi_{12}^\top > 0$, proving the claim.
\end{proof}

\newpage 

\section[Dualization result]{Sidebar: A dualization result}
\label{sidebar:dual}
Dualization is an important concept in robust control, see e.g. \cite{Scherer2001,Scherer1999}. In this paper, we provide a dualization result in Lemma~\ref{l:dualizationresult}. It is instrumental to prove Proposition~\ref{p:diss dual} that dualizes the dissipation inequality, and Corollary~\ref{c:eqnoisemodels} stated below, that relates the noise models $\mathcal{N}_1$ and $\mathcal{N}_2$. These two results can also by obtained from the dualization lemma in \cite[Lem. 4.9]{Scherer1999} by choosing appropriate complementary subspaces $\mathcal{U}$ and $\mathcal{V}$ in that result. Here, we give an alternative proof by exploiting inertia properties in Lemma~\ref{l:dualizationresult}.

\begin{lemma}
\label{l:dualizationresult}
Consider the nonsingular symmetric matrices
$$
\Psi=\bbm \Psi_{11}&\Psi_{12}\\\Psi_{12}^\top&\Psi_{22}\ebm \text{ and } \Xi=\bbm 0 & -I_r\\I_q & 0\ebm
\Psi\inv
\bbm 0 & -I_q\\I_r & 0\ebm,
$$
where $\Psi_{11}=\Psi_{11}^\top\in\R^{q\times q}$, $\Psi_{12}\in\R^{q\times r}$, and $\Psi_{22}=\Psi_{22}^\top\in\R^{r\times r}$. For any $R\in\R^{r\times q}$ we have that
\begin{equation}
\label{inertiarel}
\In\left( \bbm
I \\ R
\ebm^\top
\Psi
\bbm
I \\ R
\ebm \right) + \In\left( - \Psi^{-1}\right) = \In \left( \bbm
I \\ R^\top
\ebm^\top
\Xi
\bbm
I \\ R^\top
\ebm \right) + (q,0,q).
\end{equation}
In particular, if $\In(\Psi) = (r,0,q)$ then 
\begin{equation}
\label{equiv}
\bbm
I \\ R
\ebm^\top
\Psi
\bbm
I \\ R
\ebm \geq 0 \iff \bbm
I \\ R^\top
\ebm^\top
\Xi
\bbm
I \\ R^\top
\ebm \geq 0.
\end{equation}
\end{lemma}
\begin{proof}
Let
$$
\begin{bmatrix}
\hat{\Psi}_{11} & \hat{\Psi}_{12} \\
\hat{\Psi}_{12}^\top & \hat{\Psi}_{22}
\end{bmatrix} := - \Psi^{-1}
$$
where $\hat{\Psi}_{11}\in\R^{q\times q}$, $\hat{\Psi}_{12}\in\R^{q\times r}$, and $\hat{\Psi}_{22}\in\R^{r\times r}$. Also let $R\in\R^{r\times q}$ and define
$$
M_R := \begin{bmatrix}
0 & I & R^\top \\
I & \hat{\Psi}_{11} & \hat{\Psi}_{12} \\
R & \hat{\Psi}_{12}^\top & \hat{\Psi}_{22}
\end{bmatrix}.
$$
By Haynsworth's inertia theorem (see \cite[Fact 6.5.5]{Bernstein2011}) we have
\begin{equation}
\label{inertiaMx1}
\In(M_R) = \In(-\Psi^{-1}) + \In\left(\begin{bmatrix}
I \\ R
\end{bmatrix}^\top \Psi \begin{bmatrix}
I \\ R 
\end{bmatrix}\right).
\end{equation}
Next, we define
$$
N := \begin{bmatrix}
0 & I \\ I & \hat{\Psi}_{11}
\end{bmatrix}.
$$
Note that 
$$
N^{-1} = \begin{bmatrix}
-\hat{\Psi}_{11} & I \\ I & 0
\end{bmatrix}.
$$
By \cite[Lemma 5.1]{Maddocks1988} the matrix $N$ has inertia $\In(N) = (q,0,q)$. We also have that the Schur complement of $M_R$ with respect to $N$ is given by
\begin{align*}
\hat{\Psi}_{22} - \begin{bmatrix}
R & \hat{\Psi}_{12}^\top 
\end{bmatrix}
\begin{bmatrix}
-\hat{\Psi}_{11} & I \\ I & 0
\end{bmatrix}
\begin{bmatrix}
R^\top \\ \hat{\Psi}_{12}
\end{bmatrix} &= \\
\hat{\Psi}_{22} + R \hat{\Psi}_{11} R^\top - R\hat{\Psi}_{12} - \hat{\Psi}_{12}^\top R^\top &= \\
\begin{bmatrix}
I \\ R^\top 
\end{bmatrix}^\top 
\begin{bmatrix}
\hat{\Psi}_{22} & - \hat{\Psi}_{12}^\top \\
-\hat{\Psi}_{12} & \hat{\Psi}_{11}
\end{bmatrix}
\begin{bmatrix}
I \\ R^\top 
\end{bmatrix} &= \\
\begin{bmatrix}
I \\ R^\top 
\end{bmatrix}^\top \Xi \begin{bmatrix}
I \\ R^\top 
\end{bmatrix}&.
\end{align*}
This implies that
\begin{equation}
\label{inertiaMx2}
\In(M_R) = \In(N) + \In\left(\begin{bmatrix}
I \\ R^\top 
\end{bmatrix}^\top \Xi \begin{bmatrix}
I \\ R^\top 
\end{bmatrix}\right).
\end{equation}
By combining \eqref{inertiaMx1} and \eqref{inertiaMx2} we obtain \eqref{inertiarel}, as desired. 
Finally, suppose that $\In(\Psi) = (r,0,q)$. This implies that $\In(-\Psi^{-1}) = (q,0,r)$ and thus 
\begin{equation*}
\In\left( \bbm
I \\ R
\ebm^\top
\Psi
\bbm
I \\ R
\ebm \right) = \In \left( \bbm
I \\ R^\top
\ebm^\top
\Xi
\bbm
I \\ R^\top
\ebm \right) + (0,0,q-r).
\end{equation*}
This implies \eqref{equiv} which proves the lemma. 
\end{proof}
\begin{corollary}
\label{c:eqnoisemodels}
Suppose that Assumption (A2) holds. Then the noise models $\mathcal{N}_1$ and $\mathcal{N}_2$ are equivalent in the following sense: $\mathcal{N}_1 = \mathcal{N}_2$ if the partitioned matrices $\Phi$ and $\Theta$ in \eqref{eq:N1} and \eqref{eq:N2} are related by
$$
\begin{bmatrix}
\Phi_{11} & \Phi_{12} \\ \Phi_{12}^\top & \Phi_{22}
\end{bmatrix} = \bbm 0 & -I\\I & 0\ebm
\begin{bmatrix}
\Theta_{11} & \Theta_{12} \\ \Theta_{12}^\top & \Theta_{22}
\end{bmatrix}\inv
\bbm 0 & -I\\I & 0\ebm.
$$
\end{corollary}
\begin{proof}
By Assumption (A2) and Lemma~\ref{l:bounded with nonempty interior} we have $\In(\Theta) = (n+p,0,T)$. The corollary now follows readily from Lemma~\ref{l:dualizationresult}.
\end{proof}

\newpage 

\section[summary]{Article summary}

The concept of dissipativity, as introduced by Jan Willems, is one of the cornerstones of systems and control theory. Typically, dissipativity properties are verified by resorting to a mathematical model of the system under consideration. In this paper, we aim at assessing dissipativity by computing storage functions for linear systems directly from measured data. As our main contributions, we provide conditions under which dissipativity can be ascertained from a finite collection of noisy data samples. Three different noise models will be considered that can capture a variety of situations, including the cases that the data samples are noise-free, the energy of the noise is bounded, or the individual noise samples are bounded. All of our conditions are phrased in terms of data-based linear matrix inequalities, which can be readily solved using existing software packages. 

\newpage

\section[Biographies]{Author biographies}

\subsection{Henk van Waarde}

Henk J. van Waarde obtained the master degree (summa cum laude) and Ph.D. degree (cum laude) in Applied Mathematics from the University of Groningen in 2016 and 2020, respectively. Currently, he is a postdoctoral researcher at the University of Cambridge, UK. He was also a visiting researcher at the University of Washington, Seattle in 2019-2020. His research interests include data-driven control, system identification and identifiability, networks of dynamical systems, robust and optimal control, and the geometric theory of linear systems. He was the recipient of the 2020 Best Ph.D. Thesis Award of the Dutch Institute of Systems and Control.

\subsection{Kanat Camlibel}

M. Kanat Camlibel received the Ph.D. degree in mathematical theory of systems and control from Tilburg University in 2001. He is currently a full professor at the Bernoulli Institute for Mathematics, Computer Science, and Artificial Intelligence, University of Groningen where he served as an assistant/associate professor between 2007 and 2020. From 2001 to 2007, he held post-doctoral/assistant professor positions with the University of Groningen, Tilburg University, and Eindhoven Technical University. His research interests include differential variational inequalities, complementarity problems, optimization, piecewise affine dynamical systems, switched linear systems, constrained linear systems, multiagent systems, model reduction, geometric theory of linear systems, and data-driven control.

\subsection{Paolo Rapisarda}

Paolo Rapisarda received the Laurea (M.Sc.) degree in Computer Science from the University of Udine, Italy, and the Ph.D. degree in Mathematical Systems and Control Theory at the University of Groningen, The Netherlands, working under the supervision of J.C. Willems and H.L. Trentelman. He is currently a Full Professor in the Vision, Learning and Control Group of the University of Southampton, U.K. He is Associate Editor of the IEEE Transactions on Automatic Control, of Multidimensional Systems and Signal Processing, and of the IMA Journal of Mathematical Control and Information.  Further information can be obtained at http://www.ecs.soton.ac.uk/people/pr3. 

\subsection{Harry Trentelman}

Harry L. Trentelman is a professor of Systems and Control at the Bernoulli Institute for Mathematics, Computer Science, and Artificial Intelligence of the University of Groningen in The Netherlands. From 1991 to 2008, he served as an associate professor and later as an adjoint professor at the same institute. From 1985 to 1991, he was an assistant
professor, and later, an associate professor at the Mathematics Department of the University of Technology at Eindhoven, The Netherlands. He obtained his Ph.D. degree in Mathematics at the University of Groningen in 1985 under the
supervision of Jan C. Willems. His research interests are the behavioral approach to systems and control, robust control, model reduction, multi-dimensional linear systems, hybrid systems, analysis, control and model reduction of networked
systems, and the geometric theory of linear systems. He is a co-author of the textbook Control Theory for Linear Systems (Springer, 2001). Dr. Trentelman is past senior editor of the IEEE Transactions on Automatic Control, and past associate editor of Automatica, SIAM Journal on Control and Optimization and Systems and Control Letters. He is a Fellow
of the IEEE.

\end{document}